\documentclass[fleqn,11pt]{SelfArx} 

\usepackage[english]{babel} 
\usepackage{mathpazo}
\usepackage{lipsum} 

\usepackage{amsmath,amssymb, amsthm, mathrsfs}
\usepackage{graphicx}	
\usepackage{subfigure}
\usepackage{float} 
\usepackage{amsfonts, dsfont}
\usepackage{authblk}
\usepackage{algorithm}
\usepackage[noend]{algpseudocode}

\newcommand{\N}{\mathbb{N}}

\newcommand{\frall}{\; \forall \;}

\renewcommand{\epsilon}{\varepsilon}
\newcommand{\mbf}{\mathbf}

\newtheorem{theorem}{Theorem}
\newtheorem{lemma}{Lemma}

\graphicspath{{/Users/traylr/Documents/Papers/generalizedMultinomial/Figures/}}

\definecolor{color1}{HTML}{003366} 
\definecolor{color2}{RGB}{170, 170,170} 
\colorlet{color3}{color1!70!}

\usepackage{hyperref} 
\hypersetup{hidelinks,colorlinks,breaklinks=true,urlcolor=color2,citecolor=color1,linkcolor=color3,bookmarksopen=false,pdftitle={Title},pdfauthor={Author}}

\JournalInfo{Dell EMC Technical Reports, Vol. 1, 2017} 
\Archive{Original Research article} 

\PaperTitle{A Generalized Multinomial Distribution from Dependent Categorical Random Variables} 

\Authors{Rachel Traylor\footnote{\textit{Office of the CTO, Dell EMC}} , Ph.D.} 

\newcommand{\keywordname}{Keywords} 

\Abstract{ 			Categorical random variables are a common staple in machine learning methods and other applications across disciplines. Many times, correlation within categorical predictors exists, and has been noted to have an effect on various algorithm effectiveness, such as feature ranking and random forests. We present a mathematical construction of a sequence of identically distributed but dependent categorical random variables, and give a generalized multinomial distribution to model the probability of counts of such variables. }

\begin{document}

\flushbottom 

\maketitle 

\tableofcontents 

\thispagestyle{empty} 


\section*{Introduction} 

\addcontentsline{toc}{section}{Introduction} 

Bernoulli random variables are invaluable in statistical analysis of phenomena having binary outcomes, however, many other variables cannot be modeled by only two categories. Many topics in statistics and machine learning rely on categorical random variables, such as random forests and various clustering algorithms.~\cite{ nicodemus,grozavu}. Many datasets exhibit correlation or dependency among predictors as well as within predictors, which can impact the model used.~\cite{nicodemus, tolosi}. This can result in unreliable feature ranking~\cite{tolosi}, and inaccurate random forests~\cite{nicodemus}. 

Some attempts to remedy these effects involve Bayesian modeling~\cite{higgs} and various computational and simulation methods~\cite{tannenbaum}. In particular, simulation of correlated categorical variables has been discussed in the literature for some time.~\cite{biswas,  ibrahim,lee}. Little research has been done to create mathematical framework of correlated or dependent categorical variables and the resulting distributions of functions of such variables. 

Korzeniowski~\cite{Korzeniowski2013} studied dependent Bernoulli variables, formalizing the notion of identically distributed but dependent Bernoulli variables and deriving the distribution of the sum of such dependent variables, yielding a Generalized Binomial Distribution.

In this paper, we generalize the work of Korzeniowski~\cite{Korzeniowski2013} and formalize the notion of a sequence of identically distributed but dependent categorical random variables. We then derive a Generalized Multinomial Distribution for such variables and provide some properties of said distribution. We also give an algorithm to generate a sequence of correlated categorical random variables.

\section{Background}

\begin{figure}[t]
	\centering
	\includegraphics[scale = 0.5]{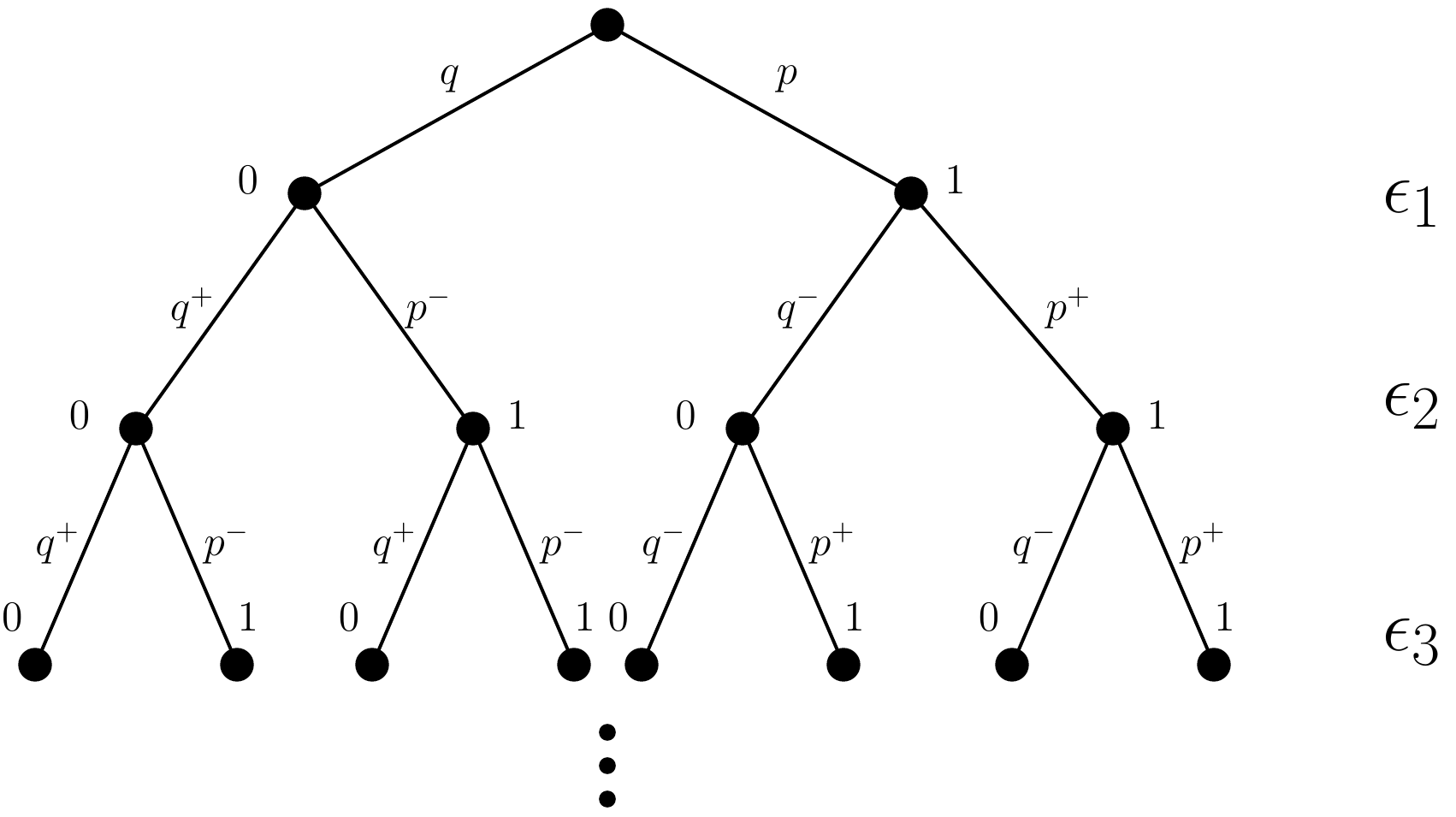}
	\caption{First-Dependence Tree Structure}
	\label{fig: Binary Tree Orig}
\end{figure}

Korzeniowski defined the notion of dependence in a way we will refer to here as \textit{dependence of the first kind} (FK dependence). Suppose $(\epsilon_{1},...,\epsilon_{N})$ is a sequence of Bernoulli random variables, and $P(\epsilon_{1} = 1) = p$. Then, for $\epsilon_{i}, i \geq 2$, we weight the probability of each binary outcome toward the outcome of $\epsilon_{1}$, adjusting the probabilities of the remaining outcomes accordingly. 

Formally, let $0 \leq \delta \leq 1$, and $q = 1-p$. Then define the following quantities
\begin{equation}
\begin{aligned}
p^{+} := P(\epsilon_{i} = 1 | \epsilon_{1} = 1) = p + \delta q &\qquad p^{-} := P(\epsilon_{i} = 0 | \epsilon_{1} = 1) = q - \delta q \\
q^{+} := P(\epsilon_{i} = 1 | \epsilon_{1} = 0) = p - \delta p  &\qquad q^{-} := P(\epsilon_{i} = 0 | \epsilon_{1} = 0) = q + \delta p
\end{aligned} 
\end{equation}

Given the outcome $i$ of $\epsilon_{1}$, the probability of outcome $i$ occurring in the subsequent Bernoulli variables $\epsilon_{2}, \epsilon_{3},...\epsilon_{n}$ is $p^{+}, i = 1$ or $q^{+}, i=0$. The probability of the opposite outcome is then decreased to $q^{-}$ and $p^{-}$, respectively. 

Figure~\ref{fig: Binary Tree Orig} illustrates the possible outcomes of a sequence of such dependent Bernoulli variables. Korzeniowski showed that, despite this conditional dependency, $P(\epsilon_{i} = 1) = p \frall i$. That is, the sequence of Bernoulli variables is identically distributed, with correlation shown to be 
\[Cor(\epsilon_{i}, \epsilon_{j}) = \begin{cases} \delta, & i=1 \\
\delta^{2}, &i \neq j, i,j \geq 2
\end{cases}\]

These identically distributed but correlated Bernoulli random variables yield a Generalized Binomial distribution with a similar form to the standard binomial distribution. In our generalization, we use the same form of FK dependence, but for categorical random variables. We will construct a sequence of identically distributed but dependent categorical variables from which we will build a generalized multinomial distribution. When the number of categories $K = 2$, the distribution reverts back to the generalized binomial distribution of Korzeniowski~\cite{Korzeniowski2013}. When the sequence is fully independent, the distribution reverts back to the independent categorical model and the standard multinomial distribution, and when the sequence is independent and $K=2$, we recover the standard binomial distribution. Thus, this new distribution represents a much larger generalization than prior models. 

\section{Construction of Dependent Categorical Variables}
\label{sec: const of dependent categorical vars}

\begin{figure}[H]
	\centering
	\includegraphics[scale=0.85]{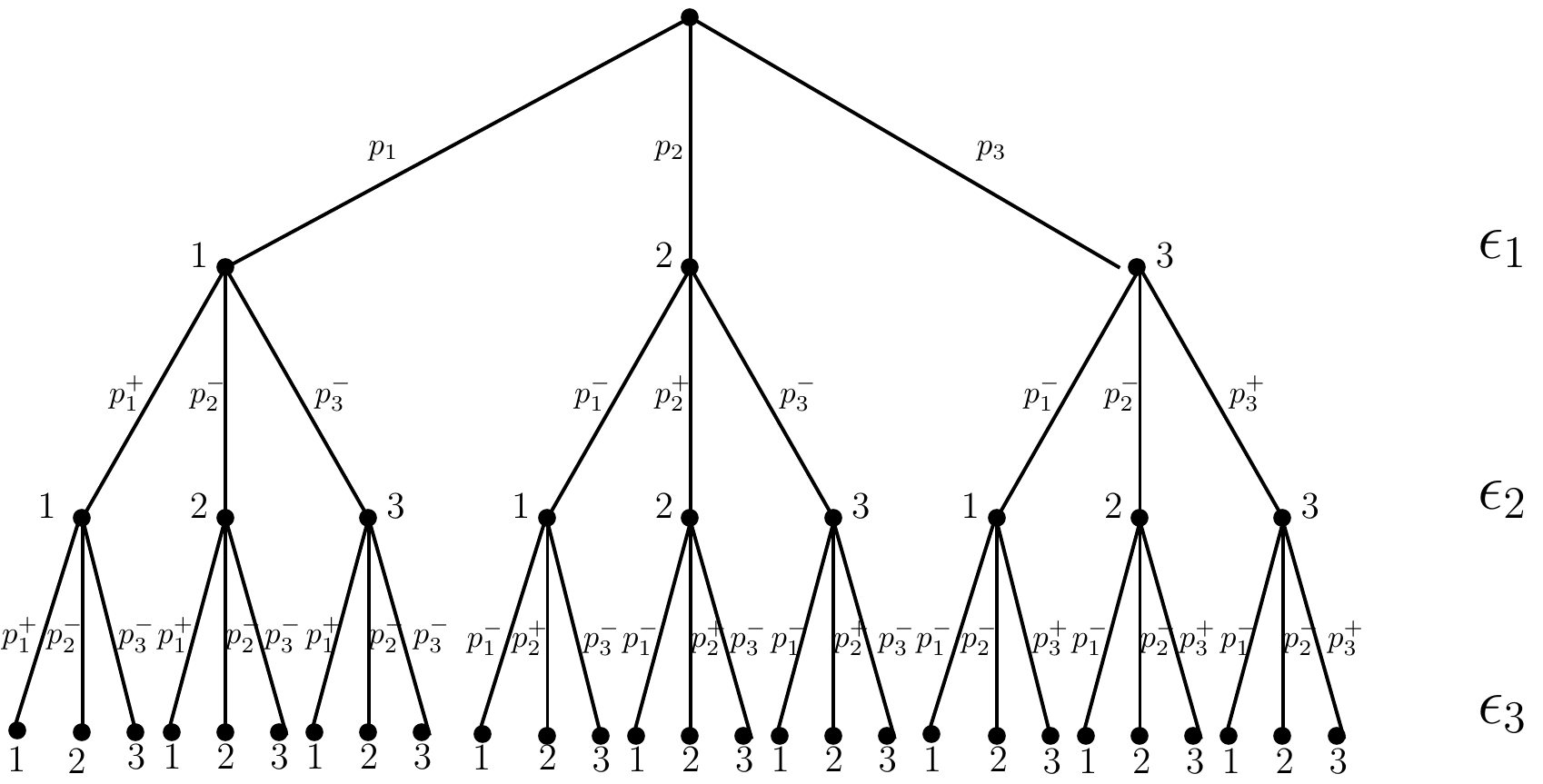}
	\caption{Probability distribution at $N=3$ for $K=3$}
	\label{fig: Probability distribution at N=3 for K=3}
\end{figure}
Suppose each categorical variable has $K$ possible categories, so the sample space $S = \{1,...,K\}$\footnote{These integers should not be taken as ordered or sequential, but rather as character titles of categories.} The construction of the correlated categorical variables is based on a probability mass distribution over $K-$adic partitions of $[0,1]$. We will follow graph terminology in our construction, as this lends a visual representation of the construction. We begin with a parent node and build a $K$-nary tree, where the end nodes are labeled by the repeating sequence $(1,...,K)$. Thus, after $N$ steps, the graph has $K^{N}$ nodes, with each node labeled $1,...,K$ repeating in the natural order and assigned injectively to the intervals 
$\left(0,\tfrac{1}{K^{N}}\right],\left(\tfrac{1}{K^{n}},\tfrac{2}{K^{N}}\right],...
,\left(\tfrac{K^{N}-1}{K^{n}},1\right]$. Define

\begin{equation}
\begin{aligned}
&\epsilon_{N} = i \hspace{1cm} \text{ on } \left(\tfrac{Kj}{K^{N}}, \tfrac{Kj+i}{K^{N}}\right] \\
&\epsilon_{N} = K \hspace{1cm} \text{on } \left(\tfrac{K(j+1)-1}{K^{N}}, \tfrac{K(j+1)}{K^{N}}\right]
\end{aligned}
\label{eq: epsilon def 1}
\end{equation}
where $i = 1,...,K-1$, and $j = 0,1,...,K^{N-1}-1$. An alternate expression for~\eqref{eq: epsilon def 1} is 
\begin{equation}
\begin{aligned}
\epsilon_{N} =i  &\text{ on } \left(\tfrac{l-1}{K^{N}}, \tfrac{l}{K^{N}}\right], & i \equiv l \mod K , & \quad i = 1,...,K-1  \\
\epsilon_{N} = K  &\text{ on } \left(\tfrac{l-1}{K^{N}}, \tfrac{l}{K^{N}}\right], & 0 \equiv l \mod K
\end{aligned}
\label{eq: epsilon def 2}
\end{equation}

To each of the branches of the tree, and by transitivity to each of the $K^{N}$ partitions of $[0,1]$, we assign a probability mass to each node such that the total mass is 1 at each level of the tree in a similar manner to~\cite{Korzeniowski2013}.

Let $0 < p_{i} < 1$, $i=1,...,K$ such that $\sum_{i=1}^{K}p_{i}=1$, and let $0 \leq \delta \leq 1$ be the \textit{dependency coefficient}. Define

$p_{i}^{+} := p_{i} + \delta\sum_{l \neq i}p_{l}$, 
$p_{i}^{-} = p_{i} - \delta p_{i}$ 

for $i=1,...,K$. These probabilities satisfy two important criteria:

\begin{lemma}
	\hspace{1in} 
	\begin{itemize}
		\item $ \sum_{i=1}^{K}p_{i} = p_{i}^{+} + \sum_{l \neq i}p_{l}^{-} = 1$
		\item $p_{i}p_{i}^{+} + p_{i}^{-}\sum_{l \neq i}p_{l} = p_{i}$ 
	\end{itemize}
	\label{lemma: criterion for pi+ pi-}
\end{lemma}

\begin{proof}
	The first is obvious from the definitions of $p_{i}^{+/-}$ above. The second statement follows clearly from algebraic manipulation of the definitions: 
	$p_{i}p_{i}^{+} + p_{i}^{-}\sum_{l\neq i}p_{l}   
	= p_{i}^{2} + p_{i}(1-p_{i}) 
	= p_{i}$
\end{proof}

We now give the construction in steps down each level of the $K$-nary tree.\\ 
\vspace{.4cm}\\
\underline{LEVEL 1:} \\
\textit{Parent node} has mass 1, with mass split $1\cdot\prod_{i=1}^{K}p_{i}^{[\epsilon_{1} = i]}$,
 where $[\cdot]$ is an Iverson bracket. This level corresponds to a sequence $\epsilon$ of dependent categorical variables of length 1. 

\begin{table}[H]
	\centering
	\begin{tabular}{|cccc|}
		\hline
		$\epsilon_{1}$/Branch & Path & Mass at Node & Interval  \\
		\hline
		1 &\textit{parent} $\to$ 1 &  $p_{1}$& $(0,1/K]$ \\
		2 &\textit{parent} $\to$ 2 &  $p_{2}$& $(1/K,2/K]$ \\
		\vdots & \vdots & \vdots &  \vdots \\
		i &\textit{parent} $\to$ i &  $p_{i}$& $((i-1)/K,i/K]$ \\
		\vdots & \vdots & \vdots &  \vdots \\
		K &\textit{parent} $\to$ K &  Mass $p_{K}$& $((K-1)/K,1]$ \\
		\hline
	\end{tabular}
	\caption{Probability mass distribution at Level 1}
	\label{table: formal construction L1}
\end{table}

\noindent\underline{LEVEL 2:} \\
Level 2 has $K$ nodes, with $K$ branches stemming from each node. This corresponds to a sequence of length 2: $\epsilon = (\epsilon_{1}, \epsilon_{2})$. Denote $i.1$ as node $i$ from level 1. For $i=1,...,K$,

{\bf Node $\mbf{i.1}$} has mass $p_{i}$, with mass split $p_{i}\left(p_{i}^{+}\right)^{[\epsilon_{2}=i]}\prod\limits_{\substack{j=1,j\neq i}}^{K}\left(p_{j}^{-}\right)^{[\epsilon_{2} = j]}$
\begin{table}[H]
	\centering
	\begin{tabular}{|cccc|}
		\hline
		$\epsilon_{2}$/Branch & Path 		  &  Mass at Node & Interval  \\
		\hline
		1	   &$i.1 \to 1$ &  $p_{i}p_{1}^{-}$ & $\left(\frac{(i-1)K}{K^{2}},\frac{(i-1)K + 1}{K^{2}}\right]$ \\
		2 	   &$i.1 \to 2$ &  $p_{i}p_{2}^{-}$ & $\left(\frac{(i-1)K + 1}{K^{2}},\frac{(i-1)K+ 2}{K^{2}}\right]$ \\
		\vdots & \vdots       & \vdots	 		   & \vdots \\
		i  	   &$i.1 \to i$ &  $p_{i}p_{i}^{+}$& $\left(\frac{(i-1)K + (i-1)}{K^{2}},\frac{(i-1)K+ i}{K^{2}}\right]$ \\
		\vdots & \vdots 	  & \vdots 	 & \vdots \\
		K      &$i.1 \to K$ &  $p_{i}p_{K}^{-}$& $\left(\frac{iK-1}{K^{2}},\frac{iK}{K^{2}}\right]$ \\
		\hline
	\end{tabular}
	\caption{Probability mass distribution at Level II, Node i}
	\label{table: formal construction L2 Ni}
\end{table}

In this fashion, we distribute the total mass 1 across level 2. In general, at any level $r$, there are $K$ streams of probability flow at Level $r$. For $\epsilon_{1} = i$, the probability flow is given by 
\begin{equation}
p_{i}\prod_{j = 2}^{r}\left[\left(p_{i}^{+}\right)^{[\epsilon_{j} = i]}\prod_{l\neq i }\left(p_{l}^{-}\right)^{[\epsilon_{j} = l]}\right], i = 1,...,K
\label{eq: prob flow at level r}
\end{equation}
We use this flow to assign mass to the $K^{r}$ intervals of $[0,1]$ at level $r$ in the same way as above. 
We may also verify via algebraic manipulation that
\begin{equation}
p_{i} = p_{i}\left(p_{i}^{+} + \sum_{l\neq i}p_{i}^{-}\right)^{r-1} = p_{i}\sum_{\epsilon_{2},...,\epsilon_{r}}\prod_{j=2}^{r}\left[\left(p_{i}^{+}\right)^{[\epsilon_{j} = i]}\prod_{l \neq i}\left(p_{l}^{-}\right)^{[\epsilon_{j} = l]}\right]
\end{equation}
where the first equality is due to the first statement of Lemma~\ref{lemma: criterion for pi+ pi-}, and the second is due to~\eqref{eq: prob flow at level r}. 

\subsection{Example construction}
For an illustration, refer to Figure~\ref{fig: Probability distribution at N=3 for K=3}. In this example, we construct a sequence of length $N = 3$ of categorical variables with $K = 3$ categories. At each level $r$, there are $3^{r}$ nodes corresponding to $3^{r}$ partitions of the interval $[0,1]$. Note that each time the node splits into 3 children, the sum of the split probabilities is 1. Despite the outcome of the previous random variable, the next one always has three possibilities. The sample space of categorical variable sequences of length 3 has $3^{3} = 27$ possibilities. Some example sequences are $(1,3,2)$ with probability $p_{1}p_{3}^{-}p_{2}^{-}$, $(2,1,2)$ with probability $p_{2}p_{1}^{-}p_{2}^{+}$, and $(3,1,1)$ with probability $p_{3}p_{1}^{-}p_{1}^{-}$. These probabilities can be determined by tracing down the tree in Figure~\ref{fig: Probability distribution at N=3 for K=3}.  

\subsection{Properties }

\subsubsection{Identically Distributed but Dependent}
We now show the most important property of this class of sequences-- that they remain identically distributed despite losing independence. 

\begin{lemma}
	$P(\epsilon_{r} = i) = p_{i}; i = 1,...,K, r \in \N$.
	\label{lemma: inductive prob distribution}
\end{lemma}

\begin{proof}
	The proof proceeds via induction. $r=1$ is clear by definition, so we illustrate an additional base case. For $r=2$, from Lemma~\ref{lemma: criterion for pi+ pi-}, and for $i=1,...,K$,
	
	\[
	P(\epsilon_{2} = i) = P\left[\bigcup_{j=1}^{K^{2}-1}\left(\frac{Kj}{K^{2}}, \frac{Kj+1}{K^{2}}\right]\right]
	= p_{i}p_{i}^{+} + p_{i}^{-}\sum_{j\neq i}p_{j} 
	= p_{i}
	\]
	
	Then at level $r$, in keeping with the alternative expression~\eqref{eq: epsilon def 2}, we express $\epsilon_{r}$ in terms of the specific nodes at level $r$:
	
	\begin{equation}
	\epsilon_{r} = \sum_{l=1}^{K^{r}}\epsilon_{r}^{l}\mathds{1}_{\left(\tfrac{l-1}{K^{r}}, \tfrac{l}{K^{r}}\right]}, \text{ where }\epsilon_{r}^{l} = \begin{cases}
	l \mod K, & 0 \not\equiv l \mod K  \\
	K, & 0 \equiv l\mod K
	\end{cases}
	\end{equation}
	
    Let $m_{r}^{l}$ be the probability mass for $\epsilon_{r}^{l}$. Then, for $l = 1,...,K^{r}$ and $i = 1,...,K-1$,
	\begin{align}
	m_{r}^{l} &= \begin{cases}
	P(\epsilon_{r}^{l} = i), &i \equiv l \mod K \\
	P(\epsilon_{r} = K), & 0 \equiv l \mod K
	\end{cases}
	\end{align}
	
	As the inductive hypothesis, assume that for $i=1,...,K-1$,
	\[
	p_{i} = P(\epsilon_{r-1} = i) 	
	= P\left(\bigcup_{\substack{l=1\\ i \equiv l \mod K }}^{K^{r-1}}\{\epsilon_{r-1}^{l} = i\}\right) 
	= \sum_{\substack{l=1\\ i \equiv l \mod K}}^{K^{r-1}}P\left(\epsilon_{r-1}^{l} = i\right) 
	= \sum_{\substack{l=1\\ i \equiv l \mod K }}^{K^{r-1}}m_{r-1}^{l}
\]
	
	Each mass $m_{r-1}^{l}$ is split into $K$ pieces in the following way
	\begin{align}
	m_{r-1}^{l} &= \begin{cases}
	m_{r-1}^{l}\left(p_{1}^{+} + \sum_{j=2}^{K}p_{i}^{-}\right), &l = 1,...,K \\
	m_{r-1}^{l}\left(p_{2}^{+} + \sum_{j\neq 2}p_{i}^{-}\right), &l = K+1,...,2K \\
	m_{r-1}^{l}\left(p_{3}^{+} + \sum_{j\neq 3}p_{i}^{-}\right), &l = 2K+1,...,3K \\
	&\vdots\\
	m_{r-1}^{l}\left(p_{K}^{+} + \sum_{j\neq K}p_{i}^{-}\right), &l = K^{r-1}-K+1,...,K^{r-1} 
	\end{cases}
	\end{align}
	which may be written as 
	\begin{align}
	m_{r-1}^{l} &= \begin{cases}
	P(\epsilon_{r} = 1) + P(\epsilon_{r}^{l} \neq 1) = m_{r}^{l} + P(\epsilon_{r}^{l} \neq 1), &l = 1,...,K \\
	P(\epsilon_{r} = 1) + P(\epsilon_{r}^{l} \neq 1) = m_{r}^{l} + P(\epsilon_{r}^{l} \neq 1), &l = K+1,...,2K \\
	P(\epsilon_{r} = 1) + P(\epsilon_{r}^{l} \neq 1) = m_{r}^{l} + P(\epsilon_{r}^{l} \neq 1), &l = 2K+1,...,3K \\
	&\hspace{-1.5in}\vdots\\
	P(\epsilon_{r} = 1) + P(\epsilon_{r}^{l} \neq 1) = m_{r}^{l} + P(\epsilon_{r}^{l} \neq 1), &l = K^{r-1}-K+1,...,K^{r-1} 
	\end{cases}
	\end{align}
	
	Then 
	\begin{equation}
	P(\epsilon_{r} = 1) = \sum_{\substack{\xi = 1\\ 1 \equiv \xi \mod K }}^{K^{r}}m_{r}^{\xi} = \sum_{l=1}^{K}m_{r-1}^{l}p_{1}^{+} + \sum_{l=K+1}^{K^{r-1}}m_{r-1}^{l}p_{1}^{-}
	\end{equation}
	
	When $1 \equiv l \mod K$, 
	\begin{equation}
	p_{1}^{+}\left(\sum_{\xi = 0}^{K-1}m_{r-1}^{l+\xi}\right) = m_{r-1}^{l}p_{1}^{+} + m_{r-1}^{l}\left(\sum_{j=2}^{K}p_{j}^{-}\right) 
	= m_{r-1}^{l}
	\label{eq: mass split}
	\end{equation}
	Equation~\ref{eq: mass split} holds because $m_{r-1}^{l+\xi} = \frac{p_{l+\xi}^{-}}{p_{1}^{+}}$, $\xi = 0,...,K-1$, and by of Lemma~\ref{lemma: criterion for pi+ pi-}. Thus,
	
	\begin{equation}
	\begin{aligned}
	P(\epsilon_{r} = 1) = \sum_{\substack{\xi = 1\\ 1 \equiv \xi \mod K }}^{K^{r}}m_{r}^{\xi}
	= \sum_{l=1}^{K}m_{r-1}^{l}p_{1}^{+} + \sum_{l=K+1}^{K^{r-1}}m_{r-1}^{l}p_{1}^{-} 
	= \sum_{\substack{l=1\\ 1 \equiv l \mod K}}^{K}m_{r-1}^{l} + \sum_{l=K+1}^{K^{r-1}}m_{r-1}^{l} 
	= \sum_{l=1}^{K^{r-1}}m_{r-1}^{l} 
	= p_{1}
	\end{aligned}
	\end{equation}
	
	A similar procedure for $i=2,...,K$ follows and the proof is complete.
\end{proof}

\subsubsection{Pairwise Cross-Covariance Matrix }
We now give the pairwise cross-covariance matrix for dependent categorical random variables. 
\begin{theorem}[Cross-Covariance of Dependent Categorical Random Variables]
	Denote $\Lambda^{\iota,\tau}$ be the $K \times K$ cross-covariance matrix of $\epsilon_{\iota}$ and $\epsilon_{\tau}$, $\iota,\tau = 1,...,n$, defined as $\Lambda^{\iota,\tau} = E[(\epsilon_{\iota} - E[\epsilon_{\iota}])(\epsilon_{\tau} - E[\epsilon_{\tau}])]$. Then the entries of the matrix are given by
	$\Lambda^{1,\tau}_{ij} = \begin{cases}
	\delta p_{i}(1-p_{i}), & i = j \\
	-\delta p_{i}p_{j}, & i \neq j
	\end{cases}$, $\tau \geq 2$, and $\Lambda^{\iota,\tau}_{ij} = \begin{cases}
	\delta^{2}p_{i}(1-p_{i}), & i = j \\
	-\delta^{2}p_{i}p_{j}, & i \neq j
	\end{cases}$, $\tau > \iota,$ $\iota \neq 1$.
	
\end{theorem}

\begin{proof}
	The $ij$th entry of $\Lambda^{\iota,\tau}$ is given by 
	\begin{equation}
	\text{Cov}([\epsilon_{\iota} = i], [\epsilon_{\tau} = j]) = \mathds{E}[[\epsilon_{\iota} = i][\epsilon_{\tau} = j]] - \mathds{E}[[\epsilon_{\iota} = i]]\mathds{E}[[\epsilon_{\tau} = j]] 
	= P(\epsilon_{\iota} = i, \epsilon_{\tau} = j) - P(\epsilon_{\iota} = i)P(\epsilon_{\tau} = j) 
	\end{equation}
	
	Let $\iota = 1$. For $j=i$, and $\tau = 2$, 
	\begin{equation}
	\text{Cov}([\epsilon_{1} = i], [\epsilon_{2} = i]) 
	= P(\epsilon_{1} = i, \epsilon_{2} = i) - p_{i}^{2} 
	= p_{i}p_{i}^{+} - p_{i}^{2} 
	= \delta p_{i}(1-p_{i})
	\end{equation}
	For $j \neq i$ and $\tau = 2$,
	\begin{equation}
	\text{Cov}([\epsilon_{1} = i], [\epsilon_{2} = j]) 
	= P(\epsilon_{1} = i, \epsilon_{2} = j) - p_{i}p_{j} 
	= p_{i}p_{j}^{-} - p_{i}p_{j} 
	= -\delta p_{i}p_{j}
	\end{equation}
	For $\tau \neq 2$, it suffices to show that $P(\epsilon_{1} = i, \epsilon_{\tau} = j) = \begin{cases}
	p_{i}p_{i}^{+}, & j = i \\
	p_{i}p_{j}^{-}, & j \neq i
	\end{cases}$. 
	
	Starting from $\epsilon_{1}$, the tree is split into $K$ large branches governed by the results of $\epsilon_{1}$. Then at level $r$, there are $K^{r}$ nodes indexed by $l$. Each of the K large branches contains $K^{r-1}$ of these nodes. That is,
	\begin{align*}
	\epsilon_{1} = 1 \text{ branch contains nodes } & l = 1,...,K^{r-1} \\
	\epsilon_{1} = 2 \text{ branch contains nodes } & l = K^{r-1}+1,...,2K^{r-1} \\
	&\vdots \\
	\epsilon_{1} = K \text{ branch contains nodes } & l = K^{r}-K+1,...,K^{r} 
	\end{align*}
	Then $P(\epsilon_{1} = 1, \epsilon_{r}^{l} = i) = m_{r}^{l}; \quad i \equiv l \mod K, l = 1,...,K^{r-1}$. We have already shown the base case for $r=2$. So, as an inductive hypothesis, assume 
	\[P(\epsilon_{1} = 1, \epsilon_{r-1} = i) = \begin{cases}
	p_{1}p_{1}^{+}, & i = 1 \\
	p_{1}p_{i}^{-}, & i \neq 1
	\end{cases}\]
	Then we have that for $i=1$,
	\[p_{1}p_{1}^{+} = P(\epsilon_{1} = 1, \epsilon_{r-1}=1) = \sum_{\substack{l=1\\1 \equiv l \mod K}}^{K^{r-2}}m_{r-1}^{l}\]
	and for $i\neq 1$,
	\[p_{1}p_{i}^{-} = P(\epsilon_{1} = 1, \epsilon_{r-1}=i) = \sum_{\substack{l=1\\i \equiv l \mod K}}^{K^{r-2}}m_{r-1}^{l}\]
	Moving one step down the tree (still noting that $\epsilon_{1} = 1$), we have seen that the each mass $m_{r-1}^{l}$ splits as 
	\[m_{r-1}^{l} = p_{1}^{+}m_{r-1}^{l} + m_{r-1}^{l}\sum_{j=2}^{K}p_{j}^{-} = P(\epsilon_{1}=1, \epsilon_{r}=1) + m_{r-1}^{l}\sum_{j=2}^{K}p_{j}^{-}\]
	Therefore, 
	\begin{align*}
	P(\epsilon_{1}=1, \epsilon_{r}=1) &= \sum_{\substack{l=1\\ 1 \equiv l \mod K}}^{K^{r-1}}m_{r}^{l} \\
	&= \sum_{\substack{l=1\\ 1 \equiv l \mod K}}^{K^{r-2}}m_{r-1}^{l}p_{1}^{+} + \sum_{j=2}^{K}\sum_{\substack{l=1\\1\not\equiv l \mod K}}^{K^{r-2}}m_{r-1}^{l}p_{j}^{-} \\
	&= p_{1}p_{1}^{+}p_{1}^{+} + p_{1}p_{1}^{+}\sum_{j=2}^{K}p_{j}^{-} \\
	&= p_{1}p_{1}^{+}
	\end{align*}
	A similar strategy shows that $P(\epsilon_{1}=1, \epsilon_{r}=i) = p_{1}p_{i}^{-}$ and that $P(\epsilon_{1}=i, \epsilon_{r}=j) = \begin{cases}
	p_{i}p_{i}^{+}, & j =i \\
	p_{i}p_{j}^{-}, & j \neq i
	\end{cases}$.

	Next, we will compute the $ij$th entry of $\Lambda^{\iota,\tau}$, $\iota > 1, \tau > \iota$. For $\iota = 2, \tau = 3$,  
	
	\begin{align}
	P(\epsilon_{2} = i, \epsilon_{3} = j) = \begin{cases}
	p_{i}\left(p_{i}^{+}\right)^{2} + (p_{i}^{-})^{2}\sum_{j\neq i}p_{j}, & i = j \\
	p_{i}p_{i}^{+}p_{j}^{-} + p_{j}p_{i}^{-}p_{j}^{+} + \sum_{\substack{l \neq i\\l \neq j}}p_{l}p_{i}^{-}p_{j}^{-}, & i \neq j
	\end{cases}
	\label{eq: epsilon2, epsilon3}
	\end{align}
	
	This can be seen by tracing the tree construction from level 2 to level 3, with an example given in Figure~\ref{fig: Probability distribution at N=3 for K=3}.
	Next, we will show that~\eqref{eq: epsilon2, epsilon3} holds for $\tau > 3$.
	
	First, note that 
	$P(\epsilon_{2} = 1, \epsilon_{\tau}^{l} = 1) = m_{\tau}^{l}$
	for $1 \equiv l \mod K$, and $l = (\xi - 1)L^{\tau-1}+j$, where $\xi = 1,...,K$, $j = 1,...,K^{\tau-2}$. We have shown the base case where $\tau = 3$. For the inductive hypothesis, assume that 
	\[P(\epsilon_{2}=1, \epsilon_{\tau-1}=1) = p_{1}\left(p_{1}^{+}\right)^{2} + (p_{1}^{-})^{2}\sum_{j=2}^{K}p_{j}^{-}\]
	Then we have that 
	$\left(p_{1}^{-}\right)^{2}\sum_{j=2}^{K}p_{j}^{-} = \sum_{1 \equiv l \mod K}m_{\tau-1}^{l}$
	for the $l$ defined above. Then, at each $m_{\tau-1}^{l}$, we have the following mass splits:
	
	\begin{equation}
	\begin{aligned}
	m_{\tau-1}^{l} &= m_{\tau-1}^{l}p_{1}^{+} + m_{\tau-1}^{l}\sum_{j=2}^{K}p_{j}^{-}, &\xi = 1 \\
	m_{\tau-1}^{l} &= m_{\tau-1}^{l}p_{i}^{+}  + m_{\tau-1}^{l}\sum_{j\neq i}m_{\tau-1}^{l}p_{j}^{-}, & \xi = i, &\quad i = 2,...,K
	\end{aligned}
	\label{eq: mass split 2}
	\end{equation}
	
	Then
	
	\begin{align}
	P(\epsilon_{2} = 1, \epsilon_{\tau} = 1) &= \sum_{\xi =1}^{K}\sum_{\substack{l,\xi \\ 1 \equiv l \mod K}}m_{\tau}^{l} 
	\end{align} 
	
	Using the same tactic as~\eqref{eq: mass split}, we see that using~\eqref{eq: mass split 2}, adding the components, and combining the terms correctly, the proof is complete for any $\tau$. 
	The proof for $P(\epsilon_{2} = i, \epsilon_{\tau} = j)$ for any $i,j$ follows similarly, and the proof for $P(\epsilon_{\iota} = i, \epsilon_{\tau} = j)$ follows from reducing to the above proven claims.

\end{proof}

In the next section, we exploit the desirable identical distribution of the categorical sequence in order to provide a generalized multinomial distribution for the counts in each category. 

\section{Generalized Multinomial Distribution}
\label{sec: gen multinomial}
From the construction in Section~\ref{sec: const of dependent categorical vars}, we derive a generalized multinomial distribution in which all categorical variables are identically distributed but no longer independent. 

\begin{theorem}[Generalized Multinomial Distribution]
	Let $\epsilon_{1},...,\epsilon_{N}$ be categorical random variables with categories ${1,...,K}$, constructed as in Section~\ref{sec: const of dependent categorical vars}. Let $X_{i} = \sum_{j=1}^{N}[\epsilon_{j} = i], i = 1,...,K$, where $[\cdot]$ is the Iverson bracket. Denote $\mbf{X} = (X_{1},...,X_{K})$, with observed values $\mbf{x} = (x_{1},...,x_{K})$. Then
	\[P(\mbf{X} = \mbf{x}) = \sum_{i=1}^{K}p_{i}\frac{(N-1)!}{(x_{i}-1)!\prod_{j \neq i}x_{j}!}\left(p_{i}^{+}\right)^{x_{i}-1}\prod_{j \neq i}\left(p_{j}^{-}\right)^{x_{j}}\]
\end{theorem}

\begin{proof}
	For brevity, let $\epsilon_{(-1)} = (\epsilon_{2},...,\epsilon_{n})$ denote the sequence of $n$ categorical random variables with the first variable removed. Conditioning on $\epsilon_{1}$,
	\begin{align*}
	P(\mbf{X} = \mbf{x}) &= \sum_{i=1}^{K}P(\mbf{X} = \mbf{x} | \epsilon_{1} = i) \nonumber \\
	&= \sum_{i=1}^{K}p_{i}\sum_{\substack{\epsilon_{-(1)}\\ \mbf{X = x}}}\prod_{j=1}^{N}\left(p_{i}^{+}\right)^{[\epsilon_{j} = i]}\prod_{\substack{k=1\\k\neq i}}^{K}\prod_{l=1}^{N}\left(p_{l}^{-}\right)^{[\epsilon_{l} = k]} \\
	&= \sum_{i=1}^{K}p_{i}\left(\sum_{\substack{\epsilon_{2},...,\epsilon_{N}\\\mbf{X} = \mbf{x}}}\left(p_{i}^{+}\right)^{\sum_{j=1}^{N}[\epsilon_{j} = i]}\prod_{l \neq i}\left(p_{l}^{-}\right)^{\sum_{j=1}^{N}[\epsilon_{j} = l]}\right)
	\end{align*}
	Now, when $\epsilon_{1} = 1$, there are ${N-1 \choose x_{1}-1}$ combinations of the remaining $N-1$ categorical variables $\{\epsilon_{i}\}_{i=2}^{N}$ to reside in category 1, ${N-1-x_{1}-1 \choose x_{2}}$ ways the remaining $N-1-x_{1}-1$ categorical variables can reside in category 2, and so forth. Finally, there are ${N-1-x_{1}-1\sum_{i=2}^{K-1}x_{i} \choose x_{K}}$ ways the final unallocated $\{\epsilon_{i}\}$ can be in category $K$. Thus,
	
	\begin{align}
	P&\left(\mbf{X} = \mbf{x}|\epsilon_{1} = 1\right)\nonumber \\
	&\qquad= p_{1}{\scriptsize{N-1} \choose x_{1}-1}{N-1-x_{1}-1 \choose x_{2}}\cdots{N-1-x_{1}-1-\sum_{j=2}^{K-1}x_{j} \choose x_{K}}\left(p_{1}^{+}\right)^{x_{1}-1}\prod_{j=2}^{K}\left(p_{j}^{-}\right)^{x_{j}} \nonumber \\
	&\qquad = p_{1}\frac{(N-1)!}{(x_{1}-1)!\prod_{j=2}^{K}x_{j}!}\left(p_{1}^{+}\right)^{x_{1}-1}\prod_{j=2}^{K}\left(p_{j}^{-}\right)^{x_{j}}
	\end{align}
	
	Similarly, for $i=2,...,K$
	
	\begin{align}
	P&\left(\mbf{X} = \mbf{x}|\epsilon_{1} = i\right)\nonumber \\
	&\qquad= p_{i}{N-1 \choose x_{1}}\cdots{N-1-\sum_{j=1}^{i-1} \choose x_{i}}\cdots{N-1-x_{i}-1-\sum_{\substack{j\neq i, j=1}}^{K-1}x_{j} \choose x_{K}}\left(p_{1}^{+}\right)^{x_{1}-1}\prod_{j=2}^{K}\left(p_{j}^{-}\right)^{x_{j}} \nonumber \\
	&\qquad = p_{i}\frac{(N-1)!}{(x_{i}-1)!\prod_{j\neq i}x_{j}!}\left(p_{i}^{+}\right)^{x_{i}-1}\prod_{j\neq i}\left(p_{j}^{-}\right)^{x_{j}}
	\end{align}
	
	Summing completes the proof.
\end{proof}

\section{Properties}
\label{subsec:  properties}

This section details some useful properties of the Generalized Multinomial Distribution and the dependent categorical random variables. 

\subsection{Marginal Distributions}
\begin{theorem}[Univariate Marginal Distribution]
	The univariate marginal distribution of the Generalized Multinomial Distribution is the Generalized Binomial Distribution. That is, 
	\begin{equation}
	P(X_{i} = x_{i}) = q{N-1 \choose x_{i}}\left(p_{i}^{-}\right)^{x_{i}}\left(q^{-}\right)^{N-1-x_{i}} + p_{i}{N-1 \choose x_{i}-1}\left(p_{i}^{+}\right)^{x_{i}-1}\left(q^{+}\right)^{N-1-(x_{i}-1)}
	\label{eq: univariate marginal}
	\end{equation}
	where $q = \sum_{\substack{j \neq i}}p_{j}$, $q^{+} = q + \delta p_{i}$, and $q^{-} = q - \delta q$
	\label{thm: univariate marginal}
\end{theorem}

\begin{proof}
	First, we claim the following:
 $q^{+} = q + \delta p_{i} = p_{l}^{+} + \sum_{\substack{j \neq l\\ j \neq i}}p_{j}^{-}$, $l=2,...,K$. This may be justified via a simple manipulation of definitions:
	\[
	p_{l}^{+} + \sum_{j \neq l}p_{j}^{-} = p_{l} + \delta\sum_{j\neq l}(p_{j}-\delta p_{j}) + \sum_{\substack{j \neq l\\ j \neq i}}(p_{j} - \delta p_{j}) 
	= p_{l} + \sum_{\substack{j\neq l\\j \neq i}}p_{j} + \delta p_{i} 
	= q + \delta p_{i}
\]
	Similarly, $q^{-} = q - \delta q$. Thus, we may collapse the number of categories to 2: Category $i$, and everything else. Now, notice that for $l\neq i$, $\left(p_{l}^{+}\right)^{x_{l}-1}\prod\limits_{\substack{j\neq l\\j\leq i}}\left(p_{j}^{-}\right)^{x_{j}} = \left(q^{+}\right)^{N-x_{i}-1}$ for $l = 1,...,K$ and $l \neq i$. Fix $k \neq i$. Then
	\begin{align}
	p_{k}\frac{(N-1)!}{(x_{k}-1)!\prod_{j\neq k}x_{j}!}\left(p_{k}\right)^{x_{k}-1}\prod_{j\neq k}\left(p_{j}^{-}\right)^{x_{j}} 
	&= p_{k}\left(p_{i}^{-}\right)^{x_{i}}\left(q^{+}\right)^{N-x_{i}-1}\frac{(N-1)!}{x_{i}!(x_{k}-1)!\prod_{\substack{j\neq k\\j\neq i}}x_{j}!} \nonumber \\
	&= p_{k}{N-1\choose x_{i}}\left(p_{i}^{-}\right)^{x_{i}}\left(q^{+}\right)^{N-x_{i}-1}
	\end{align}
	Then
	\begin{align*}
	\sum_{i=1}^{K}p_{i}\frac{(N-1)!}{(x_{i}-1)!\prod_{j \neq i}x_{j}!}\left(p_{i}^{+}\right)^{x_{i}-1}\prod_{j \neq i}\left(p_{j}^{-}\right)^{x_{j}} &= 
	p_{i}\frac{(N-1)!}{(x_{i}-1)!\prod_{j \neq i}x_{j}!}\left(p_{i}^{+}\right)^{x_{i}-1}\left(q^{-}\right)^{N-1-(x_{i}-1)} \\
	&\qquad+ \sum_{k \neq i}p_{k}\frac{(N-1)!}{(x_{k}-1)!\prod_{j \neq k}x_{j}!}\left(p_{k}^{+}\right)^{x_{k}-1}\prod_{j \neq k}\left(p_{j}^{-}\right)^{x_{j}} \\
	&= p_{i}{N-1 \choose x_{i}-1}\left(p_{i}^{+}\right)^{x_{i}-1}\left(q^{-}\right)^{N-1-(x_{i}-1)} \\
	&\qquad+ \sum_{k \neq i}p_{k}{N-1\choose x_{i}}\left(p_{i}^{-}\right)^{x_{i}}\left(q^{+}\right)^{N-x_{i}-1} \\
	&= p_{i}{N-1 \choose x_{i}-1}\left(p_{i}^{+}\right)^{x_{i}-1}\left(q^{-}\right)^{N-1-(x_{i}-1)} \\
	&\qquad+ q{N-1\choose x_{i}}\left(p_{i}^{-}\right)^{x_{i}}\left(q^{+}\right)^{N-x_{i}-1}
	\end{align*}
	
\end{proof}

The above theorem shows another way the generalized multinomial distribution is an extension of the generalized binomial distribution. 

\subsection{Moment Generating Function}

\begin{theorem}[Moment Generating Function]
	The moment generating function of the generalized multinomial distribution with $K$ categories is given by
	\begin{equation}
	M_{\mbf{X}}(\mbf{t}) = \sum_{i=1}^{K}p_{i}e^{t_{i}}\left(p_{i}^{+}e^{t_{i}} + \sum_{j \neq i}p_{j}^{-}e^{t_{j}}\right)^{n-1}
	\label{eq: mgf multinomial}
	\end{equation}
	where $\mbf{X} = (X_{1},...,X_{K}), \mbf{t} = (t_{1},...,t_{K})$.
\end{theorem}

\begin{proof}
	By definition, $M_{\mbf{X}}(\mbf{t}) = \mathds{E}\left[e^{\mathbf{t}^{T}X}\right] = \sum_{\mathbf{X}}e^{\mbf{t}^{T}\mbf{X}}P(\mbf{X} = \mbf{x}) = \sum_{\mbf{X}}e^{\mbf{t}^{T}\mbf{X}}\sum_{i=1}^{K}p_{i}\frac{(N-1)!}{(x_{i}-1)!\prod_{j\neq i}x_{j}!}\left(p_{i}^{+}\right)^{x_{i}-1}\prod_{j \neq i}\left(p_{j}^{-}\right)^{x_{j}}$.
	Let $S_{m} = \sum_{i=1}^{m}x_{i}$, and  Expanding the above, 
	\begin{align*}
	\mathds{E}&\left[e^{\mathbf{t}^{T}X}\right] = \sum_{x_{1} = 1}^{N}\sum_{x_{2}=0}^{N-x_{1}}\cdots\sum_{x_{k-1} = 0}^{N - S_{K-2}}e^{\mbf{t}^{T}\mbf{X}}p_{1}\frac{(N-1)!}{(x_{1}-1)!\prod_{j=2}^{k-1}x_{j}!(N-\sum_{j=1}^{K-1}x_{j})!}\left(p_{1}^{+}\right)^{x_{1}-1}\left[\prod_{j=2}^{K-1}\left(p_{j}^{-}\right)^{x_{j}}\right]\left(p_{K}^{-}\right)^{N - S_{K-1}} \\
		&\quad + \sum_{i=2}^{K}\left[\sum_{x_{i}=1}^{N}\sum_{x_{1}=0}^{N-S_{1}}\cdots\sum_{x_{i-1} = 0}^{N - x_{i} - S_{i-2}}\sum_{x_{i+1}=0}^{N-S_{i}}\cdots\sum_{x_{K-1}=0}^{N-S_{K-2}}
		e^{\mbf{t}^{T}\mbf{X}}p_{i}\frac{(N-1)!}{(x_{i}-1)!\prod_{j=1,j\neq i}^{K-1}x_{j}!(N-S_{K-1})!}\right. \\
		&\qquad\qquad\times\left.\left(p_{i}^{+}\right)^{x_{i}-1}\left[\prod_{\substack{j=1\\j\neq i}}^{K-1}\left(p_{j}^{-}\right)^{x_{j}}\right]\left(p_{K}^{-}\right)^{N - S_{K-1}}\right]
	\end{align*}
Taking the first term, denoted $T_{1}$, let $y = x_{1} - 1$. Then 
\begin{align*}
T_{1} = p_{1}e^{t_{1}}\sum_{y=0}^{N-1}\sum_{x_{2}=0}^{N -1 -y}\cdots\sum_{x_{K-1} = 0}^{N -1 - y -\sum_{j=2}^{K-2}x_{j} }e^{t_{1}y + t_{2}x_{2} + ... + t_{K}x_{K}}\left(p_{1}^{+}\right)^{y}\left[\prod_{j=2}^{K-1}\left(p_{j}^{-}\right)^{x_{j}}\right]\left(p_{K}^{-}\right)^{N - 1 - y - \sum_{j =2}^{K-1}x_{j}} 
\end{align*}
The summation of the above is simply the moment generating function of a standard multinomial distribution with probabilities $\mbf{p} = (p_{1}^{+}, p_{2}^{-},...,p_{K}^{-})$. Thus,
\[T_{1} = p_{1}e^{t_{1}}\left(p_{1}^{+}e^{t_{1}} + \sum_{j=2}^{K}p_{j}^{-}e^{t_{j}}\right)^{N-1}\]
A similar procedure follows with the remaining terms, and summing finishes the proof.

\end{proof}

\subsection{Moments of the Generalized Multinomial Distribution}
Using the moment generating function in the standard way, the mean vector $\mu$ and the covariance matrix $\Sigma$ may be derived. 

\paragraph{Expected Value} The expected value of $\mbf{X}$ is given by $\mu = n\mbf{p}$ where $\mbf{p} = (p_{1},...,p_{K})$

\paragraph{Covariance Matrix} The entries of the covariance matrix are given by 
\[\Sigma_{ij} = \begin{cases}
		p_{i}(1-p_{i})(n + \delta(n-1) + \delta^{2}(n-1)(n-2)), & i = j \\
		p_{i}p_{j}(\delta(1-\delta)(n-2)(n-1)-n), & i \neq j
\end{cases}\] 
Note that if $\delta = 0$, the generalized multinomial distribution reduces to the standard multinomial distribution and $\Sigma$ becomes the familiar multinomial covariance matrix. The entries of the corresponding correlation matrix are given by 
\[\rho(X_{i},X_{j}) = -\sqrt{\frac{p_{i}p_{j}}{(1-p_{i})(1-p_{j})}}\left(\frac{n-\delta(n-1)(n-2)}{n+\delta(n-1)+\delta^{2}(n-1)(n-2)}\right)\]

If $\delta = 1$, the variance of $X_{i}$ tends to $\infty$ with $n$. This is intuitive, as $\delta = 1$ implies perfect dependence of $\epsilon_{2},...,\epsilon_{n}$ on the outcome of $\epsilon_{1}$. Thus, $X_{i}$ will either be 0 or $n$, and this spread increases to $\infty$ with $n$. 

\section{Generating a Sequence of Correlated Categorical Random Variables}
\label{sec: generating sequence of DCRVs}
For brevity, we will take the acronym DCRV for a \textbf{D}ependent \textbf{C}ategorical \textbf{R}andom \textbf{V}ariable. A DCRV sequence $\epsilon = (\epsilon_{1},...,\epsilon_{n})$ is in itself a random variable, and thus has a probability distribution. In order to provide an algorithm to generate such a sequence, we first derive this probability distribution. 

\subsection{Probability Distribution of a DCRV Sequence}
\label{subsec: pdf DCRV}
The probability distribution of the DCRV sequence $\epsilon$ of length $n$ is given formally in the following theorem. The proof follows in a straightforward fashion from the construction in Section~\ref{sec: const of dependent categorical vars} and is therefore omitted. 

\begin{theorem}[Distribution of a DCRV Sequence]
		Let $(\Sigma,\mathcal{F}, \mathds{P}) = ([0,1], \mathcal{B}, \mu)$. Let $\epsilon_{i} : [0,1] \to \{1,...,K\}$, $i = 1,...,n$, $n \in \N$ be DCRVs as defined in~\eqref{eq: epsilon def 1}. Let $\epsilon = (\epsilon_{1},...,\epsilon_{n})$ denote the DCRV sequence with observed values $e = (e_{1},...,e_{n})$. Then $\mu$ has the density
		\begin{equation}
		f(x) = \sum_{i=1}^{K^{n}}K^{n}m^{i}\mathds{1}_{\tiny((i-1)/K^{n}, i/K^{n}]}(x)
		\label{eq: mu density}
		\end{equation}
		and 
		\begin{equation}
		P(\epsilon = e) = \int\limits_{\left(\frac{i-1}{K^{n}}, \frac{i}{K^{n}}\right]}f(x)dx = m^{i}
		\label{eq: pdf of DCRV sequence}
		\end{equation}
where $m^{i}$ is the mass allocated to the interval $\left(\frac{i-1}{K^{n}}, \frac{i}{K^{n}}\right]$ by~\eqref{eq: prob flow at level r} and $i$ as the lexicographic order of $e$ in the sample space $\{1,...,K\}^{n}$  given by the relation $\frac{i}{K^{n}} = \sum_{j=1}^{K^{n}}\frac{\epsilon_{j}-1}{K^{j}}$.
\label{thm: pdf of sequence}
\end{theorem}

\subsection{Algorithm}
\label{subsec: algorithm}

We take a common notion of using a uniform random variable in order to generate the desired random variable $\epsilon$. For $\epsilon$ with distribution $F(x) = \int_{0}^{x}f(y)dy$, $f(x)$ as in~\eqref{eq: mu density}, it is clear that $F$ is invertible with inverse $F^{-1}$. Thus, $F^{-1}(U)$ for $U \sim \text{Uniform}[0,1]$ has the same distribution as $\epsilon$. Therefore, sampling $u$ from $U$ is equivalent to the sample $e = F^{-1}(u)$ from $\epsilon$.

In Section~\ref{sec: const of dependent categorical vars}, we associated $\epsilon_{n}$ to the intervals given in~\eqref{eq: epsilon def 1}

\begin{equation}
\begin{aligned}
\epsilon_{N} =i  &\text{ on } \left(\tfrac{l-1}{K^{N}}, \tfrac{l}{K^{N}}\right], & i \equiv l \mod K , & \quad i = 1,...,K-1  \\
\epsilon_{N} = K  &\text{ on } \left(\tfrac{l-1}{K^{N}}, \tfrac{l}{K^{N}}\right], & 0 \equiv l \mod K
\end{aligned}
\label{eq: epsilon def 2'}
\end{equation}

From the construction in Section~\ref{sec: const of dependent categorical vars}, each sequence has a 1-1 correspondence with the interval $\left[\frac{i-1}{K^{n}}, \frac{i}{K^{n}}\right)$ for a unique $i = 1,...,K^{n}$. The probability of such a sequence can be found using Theorem~\ref{thm: pdf of sequence}:
\[P(\epsilon = e) = F\left(\left[\frac{i-1}{K^{n}}, \frac{i}{K^{n}}\right)\right) = m^{i} = l\left([s_{i-1}, s_{i})\right)\]
where $l$ is the length of the above interval, and $s_{i} = \sum_{j=1}^{i}m^{j}$. Therefore, we have now partitioned the interval $[0,1)$ according to the distribution of $\epsilon$ bijectively to the $K-$nary partition of $[0,1)$ corresponding to the particular sequence. Thus, sampling $u \in [0,1)$ from a uniform distribution and finding the interval $[s_{i-1},s_{i})$ and corresponding $i$ will yield the unique DCRV sequence.  

\paragraph{Algorithm Strategy:} Given $u \in [0,1)$ and $n \in \N$, find the unique interval $[s_{i-1}, s_{i})$, $i= 1,...,K^{n}$ that contains $u$ by ``moving down the tree" and narrowing down the ``search interval" until level $n$ is reached. 

We provide an explicit example prior to the pseudocode to illustrate the strategy.

\subsubsection{Example DCRV Sequence Generation}

	\begin{figure}[H]
		\centering
		\includegraphics[scale=.7]{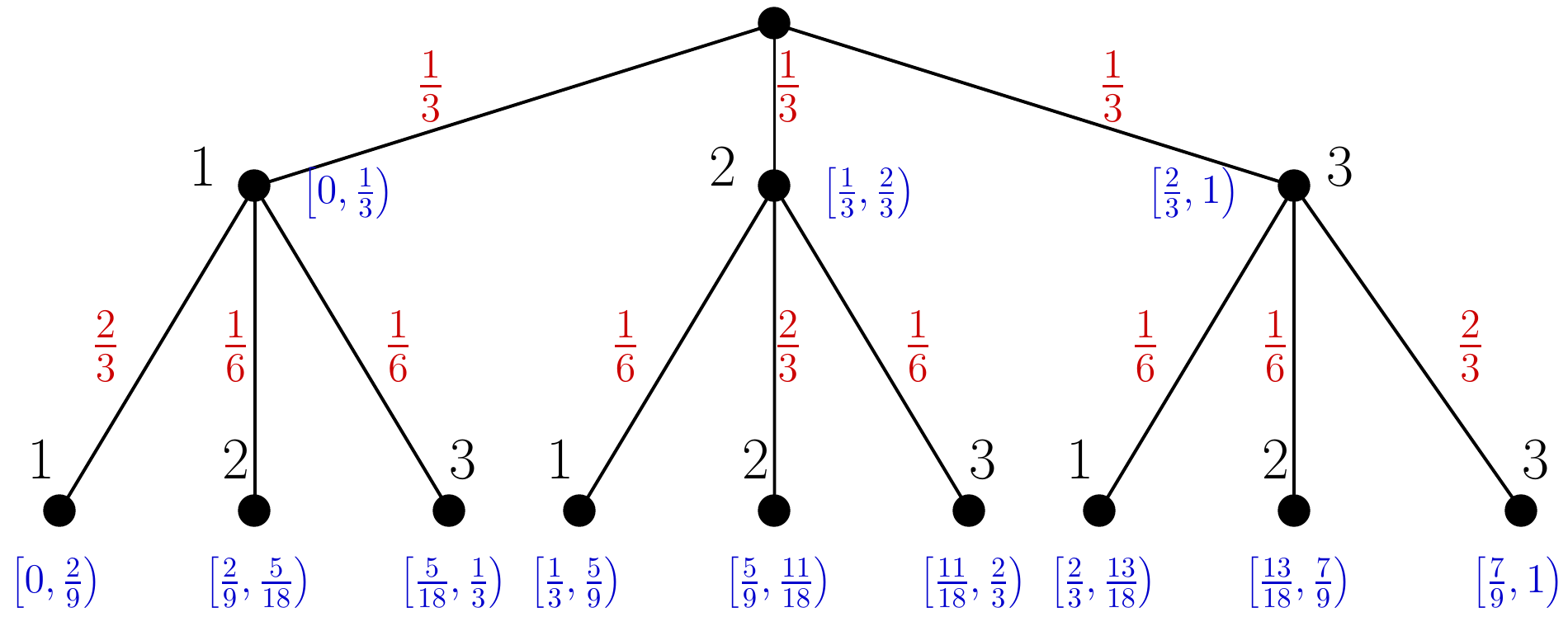}
		\caption{Probabilistic partitions of [0,1) for a DCRV sequence of length 2 with $K=3$.}
		\label{fig: algorithm illustration}
	\end{figure}

Suppose $K = 3$, $p_{1} = p_{2} = p_{3} = 1/3$, and $\delta = 1/2$, and suppose we wish to generate a DCRV sequence of length 2. Figure~\ref{fig: algorithm illustration} gives the probability flow and the corresponding probability intervals $[s_{i-1}, s_{i})$ that partition $[0,1)$ according to Theorem~\ref{thm: pdf of sequence}. Now, suppose $u = \frac{3}{4}$. We now illustrate the process of moving down the above tree to generate the sequence.

\begin{enumerate}
	\item The first level of the probability tree partitions the interval $[0,1)$ into three intervals given in Figure~\ref{fig: algorithm illustration}. $u = \frac{3}{4}$ lies in the third interval, which corresponds to $\epsilon_{1} = 3$. Thus, the first entry of $e$ is given by $e_{1} = 3$. 
	\item Next, the search interval is reduced to the interval from the first step $[2/3,1)$. We then generate the partitions of $[2/3,1)$ by cumulatively adding $p_{3}p_{i}^{-}, i = 1,2,3$ to the left endpoint $2/3$. Thus, the next partition points are 
			\begin{itemize}
				\item $2/3 + (1/3)(1/6) = 13/18$,
				\item $2/3 + (1/3)(1/6) + (1/3)(1/6) = 7/9$, and 
				\item  $2/3 + (1/3)(1/6) + (1/3)(1/6) + (1/3)(2/3) = 1$.
			\end{itemize}  
	Yielding the subintervals of $[2/3,1)$:
			\begin{itemize}
				\item $[2/3, \quad 13/18)$,
				\item $[13/18, \quad 7/9)$, and 
				\item  $[7/9,\quad 1) $.
			\end{itemize}
   We now find the interval from above that contains $u$ is the second: $[13,18, \quad 7/9)$. Thus, $\epsilon_{2} = 2$. 
\end{enumerate} 
Since we only sought a sequence of length 2, the algorithm is finished, and we have generated the sequence $e = (3,2)$. If a longer sequence is desired, we repeat step 2 until we are at level $n$. 

In general, the algorithm is given below. The $\textbf{IntervalSearch}(x,p =(p_{1},p_{2},...,p_{m}))$ procedure finds the interval $i$ built from the partitions given in the vector $p$ containing $x$ via binary search. Let $\textbf{USample}(n)$ be the procedure that samples $n$ instances of a uniformly distributed random variable. Also, let $p_{i}' = (p_{1}^{-},...,p_{i-1}^{-}, p_{i}^{+}, p_{i+1}^{-},...,p_{K})$ be the ``altered" probability vector for the categorical variables $2,...,n$ given $\epsilon_{1} = i$. 
\begin{algorithm}
	\caption{DCRV Sequence Generation}\label{alg: DCRV alg}
	\begin{algorithmic}[1]
		\Procedure{DCRVSequence}{n, $\delta$, $p = (p_{1},...,p_{K})$}
		\State $\textit{u} \gets \textbf{USample}(1)$
		\State $\textit{sequence} \gets vector(n)$
		\State $\textit{partitions} \gets \textbf{cumsum}(p_{1},...,p_{k})$
		\State $s \gets \textbf{IntervalSearch}(u,\textit{partitions})$
		\State $\textit{sequence}[1] \gets s + 1$
		\State $p_{new} \gets p'_{s+1}$
		\State $p_{prev} \gets p_{s +1}$
		\For{$i = 2, i \leq n, i+=1$}
				\State $endPoint_{l} \gets partitions[s]$
				\State $endPoint_{r} \gets partitions[s+1]$
				\State $partitions \gets endPoint_{l} + \textbf{cumsum}(p_{prev}\cdot p_{new})$
				\State $s \gets \textbf{IntervalSearch}(u,partitions)$
				\State $sequence[i] \gets s+1$
				\State $l \gets length(sequence)$
				\State $p_{prev} \gets p_{prev}\cdot p'_{sequence[l]}$
		\EndFor
		\Return sequence
		\EndProcedure
	\end{algorithmic}
\end{algorithm}

\section{Conclusion}

Categorical variables play a large role in many statistical and practical applications across disciplines. Moreover, correlations among categorical variables are common and found in many scenarios, which can cause problems with conventional assumptions. Different approaches have been taken to mitigate these effects, because a mathematical framework to define a measure of dependency in a sequence of categorical variables was not available. This paper formalized the notion of dependent categorical variables under a first-dependence scheme and proved that such a sequence is identically distributed but now dependent. With an identically distributed but dependent sequence, a generalized multinomial distribution was derived in Section~\ref{sec: gen multinomial} and important properties of this distribution were provided. An efficient algorithm to generate a sequence of dependent categorical random variables was given.

\phantomsection
\section*{Acknowledgments} 

\addcontentsline{toc}{section}{Acknowledgments} 

The author extends thanks to Jason Hathcock for his suggestions and review. 

\phantomsection
 
 \bibliographystyle{acm}
 \bibliography{GenMultinomial}

\begin{thebibliography}{1}

\bibitem{biswas}
{\sc Biswas, A.}
\newblock Generating correlated ordinal categorical random samples.
\newblock {\em Statistics and Probability Letters\/} (2004), 25--35.

\bibitem{higgs}
{\sc Higgs, M.~D., and Hoeting, J.~A.}
\newblock A clipped latent variable model for spatially correlated ordered
  categorical data.
\newblock {\em Computational Statistics and Data Analysis\/} (2010),
  1999--2011.

\bibitem{ibrahim}
{\sc Ibrahim, N., and Suliadi, S.}
\newblock Generating correlated discrete ordinal data using r and sas iml.
\newblock {\em Computer Methods and Programs in Biomedicine\/} (2011),
  122--132.

\bibitem{Korzeniowski2013}
{\sc Korzeniowski, A.}
\newblock On correlated random graphs.
\newblock {\em Journal of Probability and Statistical Science\/} (2013),
  43--58.

\bibitem{lee}
{\sc Lee, A.}
\newblock Some simple methods for generating correlated categorical variates.
\newblock {\em Computational Statistics and Data Analysis\/} (1997), 133--148.

\bibitem{nicodemus}
{\sc Nicodemus, K.~K., and Malley, J.~D.}
\newblock Predictor correlation impacts machine learning algorithms:
  implications for genomic studies.
\newblock {\em Bioinformatics\/} (2009), 1884--1890.

\bibitem{grozavu}
{\sc Nistor~Grozavu, L.~L., and Bennani, Y.}
\newblock Autonomous clustering characterization for categorical data.
\newblock {\em Machine Learning and Applications (ICMLA), 2010 Ninth
  International Conference on\/} (2010).

\bibitem{tannenbaum}
{\sc S.J.~Tannenbaum, N.H.G.~Holford, H. L. e.~a.}
\newblock Simulation of correlated continuous and categorical variables using a
  single multivariate distribution.
\newblock {\em Journal of Pharmacokinetics and Pharmacodynamics\/} (2006),
  773--794.

\bibitem{tolosi}
{\sc Tolosi, L., and Lengaurer, T.}
\newblock Classification with correlated features: unreliability of feature
  ranking and solutions.
\newblock {\em Bioinformatics\/} (2011), 1986--1994.

\end{thebibliography}


\end{document}